\documentclass[11pt]{amsart}
\usepackage[margin=1.25in]{geometry}
\usepackage[utf8]{inputenc}
\usepackage{latexsym,amsxtra,amscd,ifthen,amsmath,color, multicol,mathdots,fancyvrb,url}
\usepackage{amsfonts}
\usepackage{amssymb}
\usepackage{verbatim}
\usepackage{amsmath}
\usepackage{amsthm}
\usepackage{mathtools}
\usepackage{enumerate}
\usepackage{enumitem}
\usepackage{tikz-cd}
\usepackage[all]{xy}
 \usepackage{xcolor}
 \usepackage{adjustbox}
 \usepackage{mathrsfs}
 \usepackage{comment}
 \usepackage{todonotes}
 \usepackage{quiver}
 \usepackage{hyperref}
 \usepackage{multicol}

\definecolor{cyan(process)}{rgb}{0.0, 0.6, 1.0}
\definecolor{blue-violet}{rgb}{0.54, 0.17, 0.89}

\newcommand{\ku}{ \Bbbk}

\newcommand{\cA}{\mathcal{A}}
\newcommand{\cB}{\mathcal{B}}
\newcommand{\cC}{\mathcal{C}}
\newcommand{\cD}{\mathcal{D}}

\newcommand{\cM}{\mathcal{M}}
\newcommand{\cN}{\mathcal{N}}

\newcommand{\bF}{\mathbf{F}}
\newcommand{\bT}{\mathbf{T}}

\newcommand{\degf}[1]{\left|#1\right|}
\newcommand{\vect}{\operatorname{vec}}
\newcommand{\Irr}{\operatorname{Irr}}
\newcommand{\ot}{\otimes}

\newcommand{\fpdim}{\operatorname{FPdim}}
\newcommand{\uno}{\mathbf{1}}
\newcommand{\id}{\operatorname{id}}

\newcommand{\modaction}{\bar{\ot}}

\newcommand{\trid}{\triangleright}
\newcommand{\fiz}{\triangleleft}
\newcommand{\tridb}{\blacktriangleright}
\newcommand{\fizb}{\blacktriangleleft}

\newcommand{\stkout}[1]{\ifmmode\text{\sout{\ensuremath{#1}}}\else\sout{#1}\fi}

\title[Bicrossed products and crossed extensions]{Relation between bicrossed products and crossed extensions of fusion categories}

\usepackage{graphicx}

\newtheorem{theorem}{Theorem}[section]
\newtheorem{lemma}[theorem]{Lemma}

\newtheorem{prop}[theorem]{Proposition}
\newtheorem*{theorem-non}{Theorem}

\theoremstyle{definition}
\newtheorem{definition}[theorem]{Definition}

\makeatletter
\newenvironment{proofw}{\par
  \pushQED{\qed}%
  \normalfont \topsep6\p@\@plus6\p@\relax
  \trivlist
  \item[]\ignorespaces
}{%
  \popQED\endtrivlist\@endpefalse
}
\makeatother

\theoremstyle{remark}
\newtheorem{remark}[theorem]{Remark}

\author[M. M\"uller]{Monique M\"uller}
\address{Departamento de Matem\'atica e Estat\'istica, Universidade Federal de S\~ao
Jo\~ao del-Rei, Brazil \& Department of Mathematics, Indiana University, USA}
\email{monique@ufsj.edu.br}
\author[H. M. Pe\~na Pollastri]{H\'ector Mart\'in Pe\~na Pollastri}
\address{Department of Mathematics, Indiana University, USA}
\email{hpenapol@iu.edu}
\author[J. Plavnik]{Julia Plavnik}
\address{Department of Mathematics, Indiana University, USA \& Department of Mathematics and Data Science, Vrije Universiteit Brussel, Belgium}
\email{jplavnik@iu.edu}

\begin{document}

\begin{abstract}
We show that all crossed extensions defined by Natale can be recovered as duals of bicrossed products of fusion categories. As an application, we prove that any exact factorization between a pointed fusion category $\vect_G$ and a fusion category $\cC$ can be realized as a bicrossed product $\vect_G\bowtie \cC$. 
\end{abstract}

\maketitle

\setcounter{tocdepth}{1}

\section{Introduction}

The study of fusion categories has become of crucial importance in recent years, with the new advances in quantum computing via the Kitaev model \cite{AK} as a way to produce error-resistant quantum computers, for the description of topological states of matter, and in general as a way of formalizing 2-dimensional conformal quantum field theory as described in \cite{ms}. Then it is of great importance to find new examples and to understand how the known examples can be combined into other ones. Our emphasis is on concrete examples, where the associativity and all the maps involving the categorical structure can be displayed as a collection of matrices, the $6j$-symbols, satisfying coherence conditions. Pursuing this goal, in \cite{mp-nuestro}, the authors defined the bicrossed product of fusion categories as a way to construct and understand exact factorizations. This construction produces all possible Grothendieck rings for an exact factorization as indicated in \cite[Theorem 3.14]{mp-nuestro}, and a follow-up question is whether every exact factorization is equivalent to a bicrossed product in the categorical setting as it happens for groups \cite{Takeuchi-matched},  \cite[\S IX.1]{kassel-book} and Hopf algebras \cite{Majid-paper},\cite[\S 7.2]{Majid-book}, \cite[\S IX.2]{kassel-book}; see \cite[Question 3]{mp-nuestro}. The following theorem provides an answer when one of the factors is the category of finite-dimensional $G$-graded vector spaces for a finite group $G$, denoted as $\vect_G$.
\begin{theorem-non}[Theorem \ref{teo:teorema-2}]
    If $\cD = \vect_G\bullet \cC$ is an exact factorization of fusion categories, then $\cD$ is tensor equivalent to a bicrossed product $\vect_G\bowtie \cC$.  
\end{theorem-non}
To obtain this result, we first prove in Theorem \ref{teo:correspondencia-exact-con-bicrossed-product-con-vectG} that the extensions
\begin{align*}
    \operatorname{Rep(G)\xhookrightarrow{}} (\vect_G\bullet \cC)^*_{\cC} \xrightarrow{} \cC, &&  \operatorname{Rep(G)\xhookrightarrow{}} (\vect_G\bowtie\cC)^*_{\cC} \xrightarrow{} \cC,
\end{align*}
 associated with the exact factorizations $\vect_G\bullet \cC$ and $\vect_G\bowtie \cC$ are equivalent in the set of extensions $\operatorname{Ext}(\operatorname{Rep(G),\cC,\vect)}$ with respect to $\vect$, see Definition \ref{def:Ext}. Now the rest of the proof follows by combining two previous results concerning exact factorizations \cite{G-exact-factorization} and extensions (in the sense of exact sequences as defined in \cite{EG-exact-sequence}). 
In \cite[Corollary 5.6]{generalization-G-nonsemisimple}, the authors show that there exists a one-to-one correspondence between certain exact factorizations and extensions of fusion categories; on the other hand, a big family of extensions of fusion categories, the so-called abelian extensions, were classified in \cite[Theorem 1.1]{Natale-exact-sequence-G} as crossed extensions of fusion categories. The main result then follows from our Theorem \ref{teo:correspondencia-exact-con-bicrossed-product-con-vectG} about extensions mentioned previously.

This article is organized as follows. In Section \ref{section:Preliminaries}, we introduce the preliminaries and notation needed to state and prove the results, including an overview of bicrossed products, exact factorizations, extensions (exact sequences), and crossed extensions. In Section \ref{section:main-result}, the main results (Theorem \ref{teo:correspondencia-exact-con-bicrossed-product-con-vectG} and \ref{teo:teorema-2}) are proven.

\section*{Acknowledgements} MM was partially supported by the Association for Women in Mathematics (AWM) through the AWM Mathematical Endeavors Revitalization Program (MERP) grant. The research of JP was partially supported by the NSF grant DMS-2146392. HMPP, MM, and JP were partially supported by Simons Foundation Award 889000 as part of the Simons Collaboration on Global Categorical Symmetries. HMPP and JP would like to thank the hospitality and excellent working conditions at the Department of Mathematics at the University of Hamburg, where JP has carried out part of this research as an Experienced Fellow of the Alexander von Humboldt Foundation and HMPP as a visitor. We thank the referee for the detailed and helpful comments to improve this manuscript.

\section{Preliminaries and notation}\label{section:Preliminaries}

Let $\ku$ be an algebraically closed field of characteristic zero. The  algebraically closed assumption is required in references \cite{G-exact-factorization} and \cite{Natale-exact-sequence-G} from which we used results. For example, in \cite{Sanford}, it is explained why being algebraically closed is important. The characteristic zero hypothesis is assumed in reference \cite{Natale-exact-sequence-G}.  All vector spaces are considered to be finite-dimensional. All categories considered in this work are abelian $\ku$-linear and semisimple; we refer the reader to \cite{maclane} for further details. A \emph{fusion} category $\cC$ is a finite semisimple rigid tensor category with product $\ot$, associativity $\alpha$ and unital maps $\ell$, $r$, and the unit object $\uno$ is simple; see \cite[Definition 4.1.1]{EGNO-book}. The set of equivalence classes of simple objects is denoted by $\Irr(\cC)$.

Given a category $\cM$ and a fusion category $\cC$, we say that $\cM$ is a (right) \emph{module category over $\cC$} if there exist a bifunctor $\modaction\colon \cM \times \cC\to \cM$ and natural isomorphisms $m_{M, C, C'}\colon M\modaction(C\ot C')\to (M\modaction C)\modaction C'$, $m^0_M\colon M\modaction \uno \to M$, for all $M\in \cM$ and $C, C'\in\cC$, satisfying some conditions, see \cite[Definition 7.1.1]{EGNO-book}.

Functors are required to preserve the relevant structure. Specifically, a functor $\mathbf{F}\colon \cC \to \cD$ between categories is always $\ku$-linear and exact. If $\cC$ and $\cD$ are fusion categories, the functor $\mathbf{F}=(\mathbf{F}, \mathbf{F}^2_{X,Y}, \mathbf{F}^0)$ is also required to be tensor, where $\mathbf{F}^2_{X,Y}\colon \mathbf{F}(X)\ot \mathbf{F}(Y) \to  \mathbf{F}(X\ot Y)$ and $\mathbf{F}^0\colon \mathbf{F}(\uno)\to \uno$, see \cite[Definition 2.4.1]{EGNO-book}.

A functor $(\mathbf{F}, s)\colon \cM \to \cN$ between $\cC$-module categories for a fusion category $\cC$ is a pair consisting of a $\ku$-linear functor $\mathbf{F}\colon \cM\to \cN$ together with natural isomorphisms $s_{M,C}\colon \mathbf{F}(M\modaction C) \to \mathbf{F}(M)\modaction C$ satisfying the conditions given in \cite[Definition 7.2.1]{EGNO-book}.

Given an indecomposable $\cC$-module category $\cM$, the dual category $\cC^*_{\cM}$ of $\cC$ with respect to $\cM$ is the (strict) fusion category $\operatorname{End}_{\cC}(\cM)$ of $\cC$-module endofunctors $(\mathbf{F},s):\cM\to\cM$, with tensor product given by composition of functors and duals by adjunctions. Associated to a tensor functor $\bT\colon \cD\to \cC$ between fusion categories $\cD$ and $\cC$ there is an induced $\cD$-module structure on $\cM$ by $\bT$ and we have a tensor functor $\bT^*_{\cM}\colon \cC_\cM^*\to \cD_\cM^*$ that maps $(\bF,s)$ to $(\bF,t)$, where $t_{M,D} = s_{M,\bT(D)}$, for every $M\in\cM$ and $D\in\cD$.

\subsection{$G$-categories}\label{subsec: G-action} In this subsection, we recall  $\ku$-linear abelian categories with action by a group. A category equipped with such an action is said to be a $G$-category, and it categorifies the notion of a $G$-set. This notion is relevant for describing crossed extensions and their equivalence with certain bicrossed products.

\begin{definition}{\rm(\cite[Definition 2.7.1]{EGNO-book})}
    Let $G$ be a group. A  \emph{(right) $G$-category} is a pair $(\cC,\fiz)$, where $\cC$ is a category equipped with a \emph{right action} $\fiz$ of $G$. That is, for every $g\in G$ there exist an exact $\ku$-linear functor $-\fiz g\colon \cC \to \cC$ and there are natural isomorphisms $\mathbf{R}^2_{g,h}\colon (-\fiz g)\fiz h \to -\fiz gh$ and $\mathbf{R}^0\colon \id_\cC\to -\fiz e$ satisfying
    \begin{align}\label{action-eq1}
        &(\mathbf{R}^2_{gh,k})_{X}((\mathbf{R}^2_{g,h})_X)\fiz k)= (\mathbf{R}^2_{g,hk})_{X}(\mathbf{R}^2_{h,k})_{X\fiz g},\\
        \label{action-eq2}&(\mathbf{R}^2_{g,e})_{X} \mathbf{R}^0_{X\fiz g} = \id_{X\fiz g} = (\mathbf{R}^2_{e,g})_{X}((\mathbf{R}^0_{X})\fiz g), 
    \end{align}
    for every $X\in\cC$, $g,h,k\in G$. Similarly, we define a \emph{left action} $\trid$, that is, for every $g\in G$ there is an exact functor $g\trid -$ with natural isomorphisms $\mathbf{L}^2_{g,h}: h\trid (g\trid -)\to gh\trid -$ and $\mathbf{L}^0: \id_{\cC}\to  e\trid -$ satisfying equations analogues to the previous ones.
\end{definition}

Analogously, the notion of a $G$-functor between $G$-categories can be defined.
    \begin{definition}{\rm(\cite[Definition 2.7.1]{EGNO-book})}
    Let $\cC$ and $\cD$ be two $G$-categories. A \emph{$G$-functor} between $\cC$ and $\cD$ is a pair $(\mathbf{F}, \{\mu^g\}_{g\in G})$, where $\mathbf{F}\colon \cC \to \cD$ is a functor and $\mu^g_X\colon F(X)\fiz g \to F(X\fiz g)$ is a family of natural isomorphism such that the following equation is satisfied
    \begin{align}\label{eq:definition-G-functor}
        \mathbf{F}((\mathbf{R}^2_{g,h})_X)(\mu^h_{X\fiz g})(\mu^g_X \fiz h) &= (\mu^{gh}_X)(\mathbf{R}^2_{g,h})_{\mathbf{F}(X)}.
    \end{align}
   
    Let $(\mathbf{F}, \{\mu^g\}_{g\in G})$ and $(\mathbf{H}, \{\theta^g\}_{g\in G})$ be two $G$-functors between the $G$-categories $\cC$ and $\cD$. A \emph{$G$-natural transformation} is a natural transformation $\sigma\colon \mathbf{F} \to \mathbf{H}$ such that $\sigma_{X\fiz g}\mu_X^g=\theta^g_X(\sigma_X\fiz g)$, for all $X\in\cC, g\in G$.
\end{definition}

A natural notion on a $G$-category $\cC$ is the notion of \emph{equivariant object}. These are pairs $(X, \{u_g\}_{g\in G})$, where $X$ is an object in $\cC$ and $\{u_g\}_{g\in G}$ is a family of isomorphisms $u_g\colon X\fiz g \to X$ such that 
\begin{eqnarray}\label{eq:equivariant-object-definition}
    u_{gh}(\mathbf{R}^2_{g,h})_X & = u_h(u_g\fiz h),
\end{eqnarray} for all $g,h\in G$, see \cite[Definition 2.7.2]{EGNO-book}.
We can define a category with the equivariant objects denoted as $\cC^G$, namely the \emph{equivariant category} or  \emph{$G$-equivariantization of $\cC$}. A morphism between equivariant objects $(X,\left\{u_g\right\}_{g\in G})$ and $(Y,\left\{v_g\right\}_{g\in G})$ is a morphism $f\colon X\to Y$ such that $v_g (f\fiz g) = f u_g$, for all $g\in G$.

The following lemma is useful to construct equivariant objects.

\begin{lemma}\label{lemma:G-funtor-preserva-obj-equivariant}
    Let $\cC$ and $\cD$ be $G$-categories, $(\mathbf{F}, \theta^g)\colon \cC \to \cD$ be a $G$-functor, and $(X, \{u_g\}_{g\in G})$ be an equivariant object in $\cC$. Then $(\mathbf{F}(X), \{\mathbf{F}(u_g)\theta^g_X\}_{g\in G})$ is an equivariant object in $\cD$.
\end{lemma}
\begin{proof}
    We just need to show that the following diagram is commutative:
    \begin{equation*}
    \begin{adjustbox}{scale=.9}
                \begin{tikzcd}
            (\mathbf{F}(X)\fiz g)\fiz h \ar[r,"\theta^g_X \fiz h"] \ar[d, swap,"(\mathbf{R}^2_{g,h})_{\mathbf{F}(X)}"] & \mathbf{F}(X\fiz g)\fiz h \ar[r, "\mathbf{F}(u_g)\fiz h"] \ar[d, swap, "\theta^h_{X\fiz g}"] \ar[rd, phantom, "(ii)"] & \mathbf{F}(X)\fiz h \ar[d,"\theta^h_X"]\\
            \mathbf{F}(X)\fiz gh \ar[r, phantom, "(i)"] \ar[dr, swap, bend right, "\theta_X^{gh}"] & \mathbf{F}((X\fiz g)\fiz h) \ar[r,"\mathbf{F}(u_g \fiz h)"] \ar[d, swap, "\mathbf{F}((\mathbf{R}^2_{g,h})_X)"] \ar[rd, phantom, "(iii)"] & \mathbf{F}(X\fiz h) \ar[d, "\mathbf{F}(u_h)"]\\
            & \mathbf{F}(X\fiz gh) \ar[r, swap, "\mathbf{F}(u_{gh})"] & \mathbf{F}(X),
        \end{tikzcd}
            \end{adjustbox}
    \end{equation*}
    for all $g, h\in G$. The commutativity of the diagram $(i)$ follows from \eqref{eq:definition-G-functor}, $(ii)$ follows from the naturality of $\theta^h$, and $(iii)$ from applying $\mathbf{F}$ to \eqref{eq:equivariant-object-definition}.
\end{proof}

\subsubsection*{Module categories over $\vect_G$} Let $G$ be a finite group. We denote by $\vect_G$ the category of finite-dimensional $G$-graded vector spaces.  This is a fusion category with isomorphism classes of simple objects labeled by elements $g\in G$, where $g$ denotes the one dimensional vector space concentrated in degree $g$. The tensor product of simple objects is given by the product of the group $g\ot h = gh$, and the associativity isomorphisms are the ones from vector spaces.

There is an intimate relation between $G$-actions on a category $\cM$ and $\vect_G$-module category structures on $\cM$. Indeed, let $\cM$ be a module category over $\vect_G$, then the following lemma gives a structure of a $\ku$-linear $G$-category on $\cM$. We omit the proof because it is straightforward.

\begin{lemma}\label{lemma:G-cat-G-functor}
	Let $\cM$ be a (right) $\vect_G$-module category. Then there is an induced (right) $G$-action on $\cM$ defined as follows:
	\begin{enumerate}[leftmargin=*,label=\rm{(\roman*)}]
	    \item the functor $-\bar{\fiz} g $ is given by $ - \modaction g$  for each $g\in G$,
		\item the natural isomorphism $(\mathbf{R}^2_{g,h})_M \colon (M\bar{\fiz} g)\bar{\fiz} h \to M\bar{\fiz} gh$ is given by $m_{M, g, h}^{-1}$,
        \item the isomorphism $\mathbf{R}^0_M\colon M \to M\bar{\fiz} e$ is given by $(m^0_M)^{-1}$.
	\end{enumerate}
Furthermore, a $\vect_G$-module functor $(\mathbf{F}, s)$ induces a $G$-functor $(\mathbf{F},\theta^g)$, with $\theta^g_M = (s_{M,g})^{-1}$.
\end{lemma}

\subsection{Exact factorizations and extensions}
In this subsection, we recall definitions and results of exact factorizations and extensions of fusion categories. We start with exact factorizations of groups. An \emph{exact factorization} of a group $\Sigma$ is a pair $(G, \Gamma)$ of subgroups of $\Sigma$ such that $\Sigma = G \Gamma$ and $G \cap \Gamma = \{e\}$.  We can describe this notion in terms of a \emph{matched pair} of groups, which is a collection 
$(G, \Gamma, \tridb, \fizb)$ such that $G$ and $\Gamma$ are groups, and $\tridb$ and $\fizb$ are left and right actions 
$\xymatrix{\Gamma & \Gamma \times G \ar  @{->}[r]^{\quad\tridb}\ar  @{->}[l]_{\fizb\quad } & G }$ satisfying
\begin{align*}
	(kt) \fizb g &= \left(k \fizb (t \tridb g)\right)(t \fizb g) & \text{ and } &&
	k \tridb (g h) &=\left(k \tridb g\right)\left((k \fizb g) \tridb h \right),
 \end{align*}
for all $k, t \in \Gamma$, $g, h \in G$.

A matched pair of groups defines an exact factorization as follows. 
Given a matched pair of groups $(G, \Gamma, \tridb, \fizb)$, we define  $G \bowtie 
\Gamma$ as the set $G \times \Gamma$ equipped with the product
\begin{align*}
	(h, k)(g , t) &= \left(h (k \tridb g), (k \fizb g) t\right),
	& k, t &\in \Gamma, \ g, h \in G.
\end{align*}
The group $G\bowtie \Gamma$ is known as the \emph{Zappa–Sz\'ep product} or \emph{bicrossed product} of the groups $G$ and $\Gamma$, and it is an exact factorization of them. Furthermore, every exact factorization of $G$ and $\Gamma$ can be obtained in this way; see \cite[Proposition 2.4]{Takeuchi-matched}.

 Let $\cB$ be a fusion category, and let $\cA$, $\cC$ be fusion subcategories of $\cB$. The intersection category $\cA\cap\cC$ is the full subcategory of $\cB$ whose objects are in $\cA$ and $\cC$. The category $\cB$ is an \emph{exact factorization} of $\cA$ and $\cC$ if $\cA\cap\cC \simeq \vect$, the category of finite-dimensional vector spaces over $\ku$, and $\fpdim(\cB) = \fpdim(\cA)\fpdim(\cC)$, see \cite[Definition 3.4]{G-exact-factorization}. In that case, we write $\cB = \cA \bullet \cC$. Next, we will define the set $\operatorname{ExFac}(\cA,\cC)$.
 
\begin{definition}
     Let $\cA_1 = \cA_1'\bullet \cA_1''$ and $\cA_2 = \cA_2' \bullet \cA_2''$ be exact factorizations of fusion categories, and $\iota_i'\colon \cA_i'\hookrightarrow\cA_i$, $\iota_i''\colon \cA_i''\hookrightarrow\cA_i$ be the corresponding inclusion functors. We say that the exact factorizations $\cA_1$ and $\cA_2$  are \emph{equivalent} if there exist tensor equivalences $\mathbf{T}\colon \cA_1\to \cA_2$, $\mathbf{T}'\colon \cA_1'\to \cA_2'$ and $\mathbf{T}''\colon \cA_1''\to \cA_2''$ such that the following diagram commutes up to a monoidal natural transformation
    \begin{equation*}
    \begin{tikzcd}
        \cA_1' \ar[r, hook, "\iota_1'"]\ar[d, "\mathbf{T'}"] & \cA_1 \ar[d, "\mathbf{T}"] & \cA_1''\ar[d, "\mathbf{T}''"] \ar[l, swap, hook', "\iota_1''"] \\
\cA_2'\ar[r, hook, "\iota_2'"] & \cA_2 & \cA_2''\ar[l, swap, hook', "\iota_2''"].
\end{tikzcd}
    \end{equation*}
    Given $\cA$ and $\cC$ fusion categories, we denote by $\operatorname{ExFac}(\cA,\cC)$ the set of equivalence classes of exact factorizations $\cB= \cA\bullet \cC$.
\end{definition}

 Let $\cM$ be a finite semisimple abelian category and $\operatorname{End}{(\cM)}$ be the category of right endofunctors of $\cM$. A fusion category $\cD$ is an \emph{extension of $\cC$ by  $\cA$ with respect to $\cM$} if there exist tensor functors
\begin{align}\label{exact-sequence}
    \cA\xhookrightarrow{\mathbf{\iota}} \cD \xrightarrow{\mathbf F} \cC\boxtimes\operatorname{End}{(\cM)^{op}},
\end{align}
such that 
\begin{enumerate}[leftmargin=*,label=(\roman*)]
\item $\mathbf{\iota}$ is injective, that is, $\mathbf{\iota}$ is a fully faithful embedding,
\item $\mathbf{F}$ is surjective, that is, every object $X\in\cC\boxtimes\operatorname{End}{(\cM)^{op}}$ is a subquotient of $\mathbf{F}(D)$ for some $D\in \cD$ (see equivalent conditions in \cite[Lemma 3.1]{Natale-Brugiere}),
\item\label{item3} $\iota(\cA) = \ker(\mathbf{F})$,  where $\ker(\mathbf{F})$ is the subcategory of $X\in\cD$ such that $\mathbf{F}(X)\in \operatorname{End}{(\cM)}$, and \item $\mathbf F$ is normal, that is, for any $X\in\cD$ there exists a subobject $X_0\subseteq X$ such that $\mathbf{F}(X_0)$ is the largest subobject of $\mathbf{F}(X)$ such that $\mathbf{F}(X_0)\simeq \uno\boxtimes \mathbf{G}$ for some $\mathbf{G}\in \operatorname{End}(\cM)$. 
\item The category $\cM$ admits a canonical structure of a right $\cA$-module category by taking $\rho\colon \cA\to \operatorname{End}(\cM)^{op}$ as $\rho = \mathbf{F}\circ \iota$, which is well-defined by \ref{item3}. We require $\cM$ to be indecomposable as an $\cA$-module category.
\end{enumerate}
Notice that $\mathbf{F}$ is faithful by \cite[Lemma 2.1]{Natale-Brugiere}. 

Next, we define the set $\operatorname{Ext}(\cA,\cC,\cM)$, see also \cite[\S 3]{Natale-exact-sequence-G}.

\begin{definition}\label{def:Ext}
    Two extensions of fusion categories 
\begin{align*}
     \cA_1'\xhookrightarrow{\mathbf{\iota_1}} \cA_1 \xrightarrow{\mathbf{F}_1} \cA_1''\boxtimes\operatorname{End}{(\cM_1)}^{op}, &&  \cA_2'\xhookrightarrow{\mathbf{\iota}_2} \cA_2 \xrightarrow{\mathbf{F}_2} \cA_2''\boxtimes\operatorname{End}{(\cM_2)}^{op},
\end{align*}
are \emph{equivalent} if there exist tensor equivalences $\mathbf{T}\colon \cA_1\to \cA_2$, $\mathbf{T'}\colon \cA_1'\to \cA_2'$ and $\mathbf{T''}\colon \cA_1''\to \cA_2''$, and a $\ku$-linear equivalence $\mathbf{S}\colon \cM_1\to \cM_2$ such that the following diagrams commute up to monoidal natural transformations
\begin{equation*}
    \begin{tikzcd}
        \cA_1' \ar[r, hook, "\iota_1"]\ar[d, "\mathbf{T'}"] & \cA_1\ar[r, "\mathbf{F}_1"] \ar[d, "\mathbf{T}"] & \cA_1''\boxtimes\operatorname{End}(\cM_1)^{op}\ar[d, "\mathbf{T}''\boxtimes\operatorname{End}(\mathbf{S})"] \\
\cA_2'\ar[r, hook, "\iota_2"] & \cA_2\ar[r, "{\mathbf F}_2"] & \cA_2''\boxtimes\operatorname{End}(\cM_2)^{op},       
    \end{tikzcd}
\end{equation*}
where $\rho_1$ and $\rho_2$ are the tensor functors defining the module category structure, and $\operatorname{End}(\mathbf{S})$ is the tensor functor induced by $\mathbf{S}$ that maps an endofunctor $\psi\colon \cM_1\to \cM_1$ to $\mathbf{S}\circ \psi\circ\mathbf{S}^{-1}$. Given $\cA$ and $\cC$ fusion categories and $\cM$ an exact indecomposable $\cA$-module category, we denote by $\operatorname{Ext}(\cA,\cC,\cM)$ the set of equivalence classes of extensions $\cA\xhookrightarrow{\iota}\cD\xrightarrow{\mathbf{F}} \cC\boxtimes \operatorname{End}(\cM)^{op}$.
\end{definition}

\begin{remark}\label{remark:estable-by-semisimplicity}
    The definitions of exact factorization and extensions can be done in the finite tensor category context (without the semisimplicity hypothesis). Both definitions are stable under the semisimplicity condition, that is, if both $\cA$ and $\cC$ are semisimple then the possible factorizations and extensions are semisimple too, see \cite[Corollary 3.13]{generalization-G-nonsemisimple} and \cite[Theorem 3.8]{EG-exact-sequence}. 
\end{remark}

\begin{remark}
    Notice that in the definition of $\operatorname{ExFac}(\cA,\cC)$ and $\operatorname{Ext}(\cA,\cC,\cM)$ the categories involved in the label are defined up to tensor (or module) equivalence. That is, if $\cA'$ and $\cC'$ are tensor equivalent to $\cA$ and $\cC$ respectively, then $\operatorname{ExFac}(\cA,\cC) = \operatorname{ExFac}(\cA',\cC')$, as the corresponding equivalences induce equivalence of exact factorizations, similarly with extensions.
\end{remark}

\subsection{Basak-Gelaki Correspondence}

In the work of Basak and Gelaki \cite[Theorem 5.1]{generalization-G-nonsemisimple}, the authors stated a correspondence between exact factorizations and extensions of tensor categories. They showed that this correspondence is one-to-one in some specific cases, see \cite[Corollaries 5.4 and 5.6]{generalization-G-nonsemisimple}. For a fusion category $\cA$ and an indecomposable right $\cA$-module category $\mathcal{M}$, the correspondence is given by functions 
\begin{align*}
   & \varphi_{\cA, \mathcal{M}}\colon \operatorname{Ext}(\cC,\cA,\cM)\to \operatorname{ExFac}(\cC,\cA^*_{\cM}), \text{and}\\ &\psi_{\cA,\cM}\colon \operatorname{ExFac}(\cC, \cA)\to \operatorname{Ext}(\cC,\cA^*_{\cM}, \cM)
\end{align*}
 that will be described next.

\subsubsection*{The map $\varphi_{\cA, \mathcal{M}}$}\label{subsubsect:mapa-phi}
Given an extension $\cA\xhookrightarrow{\iota}\cD\xrightarrow{\mathbf{F}} \cC\boxtimes \operatorname{End}(\cM)^{op}$, we map it to the exact factorization $\cA^*_{\cM}\xhookrightarrow{\iota_{{\cA}^*_{\cM}}} \cD^*_{\cN} \xhookleftarrow{\mathbf{F}^*} (\cC\boxtimes \operatorname{End}(\cM)^{op})^*_{\cN}$, where $\cN =  \operatorname{Ind}_{\cA}^{\cD}(\cM) = \cM\boxtimes_{\cA}\cD$
is a indecomposable $\cD$-module category (see the proof of \cite[Theorem 5.1]{generalization-G-nonsemisimple})  with structure induced by $\iota$, and $\iota_{{\cA}^*_{\cM}}$ maps an $\cA$-module functor $\mathbf{G}: \cM\to \cM$ to the induced $\cD$-module functor $ \mathbf{G}\boxtimes_{\cA}\id_{\cD} = \operatorname{Ind}_{\cA}^{\cD}(\mathbf{G}):\cN\to \cN$. It follows from the proof \cite[Theorem 5.1]{generalization-G-nonsemisimple} that $(\cC\boxtimes \operatorname{End}(\cM)^{op})^*_{\cN} \simeq \cC$. We can consider $\widetilde{\cN} = \cC\boxtimes\cM$ as a $\cC\boxtimes \operatorname{End}(\cM)^{op}$-module category with the canonical action in each components (the $\cC$-action by taking tensor product on the right, and the $\operatorname{End}(\cM)^{op}$-action by evaluation), then $\widetilde{\cN}$ is also a $\cD$-module category by restriction of $\mathbf{F}$, and $\mathbf{F}^*\colon (\cC\boxtimes \operatorname{End}(\cM)^{op})^*_{\widetilde{\cN}} \to \cD_{\widetilde{\cN}}$ just view a module functor over $\cC\boxtimes \operatorname{End}(\cM)^{op}$ as a $\cD$-module functor. By \cite[Theorem 2.9]{EG-exact-sequence}, $\widetilde{\cN}\simeq \cN$ as right $\cD$-module categories, hence $\cD^*_{\widetilde{\cN}}\simeq \cD^*_\cN$. Under this identification, we obtain the desired exact factorization by \cite[Theorem 5.1]{generalization-G-nonsemisimple}.

\subsubsection*{The map $\psi_{\cA,\cM}$}\label{subsubsect:mapa-psi}
Given an exact factorization $\cA \xhookrightarrow{\nu} \cB \xhookleftarrow{\kappa} \cC$, we map it to the extension $\cA^*_{\cM}\xhookrightarrow{\iota_{{\cA}^*_{\cM}}}\cB^*_{\cN}\xrightarrow{\kappa^*} \cC^*_{\cN}$, where $\cN =  \operatorname{Ind}_{\cA}^{\cB}(\cM) =\cM\boxtimes_{\cA}\cB$
is a $\cB$-module category with the structure induced  by $\nu$,  $\iota_{{\cA}^*_{\cM}}$ maps an $\cA$-module functor $\mathbf{G}: \cM\to \cM$ to the induced $\cB$-module functor $\mathbf{G}\boxtimes_{\cA}\id_{\cB} = \operatorname{Ind}_{\cA}^{\cB} (\mathbf{G}):\cN\to\cN$, since $\cN\simeq\cM\boxtimes \cC$ as $\cC$-module categories,   $\kappa^*\colon \cB^*_{\cN}\to  \cC^*_{\cN}$ is defined by just viewing a $\cB$-module functor $\cN\to \cN$ as a $\cC$-module functor, and $\cC^*_{\cN}\simeq \cC\boxtimes \operatorname{End}(\cM)^{op}$. By \cite[Theorem 5.1]{generalization-G-nonsemisimple}, we have the equivalent extension $\cA^*_{\cM}\xhookrightarrow{\iota_{{\cA}^*_{\cM}}}\cB^*_{\cN}\xrightarrow{\kappa^*} \cC\boxtimes \operatorname{End}(\cM)^{op}$ given an alternative expression of the image of $\psi_{\mathcal{A}, \mathcal{M}}$.

\begin{remark}
The above correspondence is one-to-one when $\cM=\vect$, see \cite[Corollary 5.6]{generalization-G-nonsemisimple}.
\end{remark}

\begin{remark}
    In \cite{generalization-G-nonsemisimple}, the correspondence is stated in a more general hypothesis; it is valid for finite tensor categories without the semisimplicity, but by Remark \ref{remark:estable-by-semisimplicity} we can restrict ourselves to the semisimple case.
\end{remark}

\subsection{Matched pairs and bicrossed product of fusion categories} In \cite[\S 4]{mp-nuestro}, we introduced the notion of bicrossed product of fusion categories as a way to construct exact factorizations. At the Grothendieck ring level, these characterize all possible exact factorizations; see \cite[Theorem 3.14]{mp-nuestro}. Let $\mathcal{A}$ and $\mathcal{C}$ be fusion categories and let $G$ and $\Gamma$ be finite groups. A \emph{matched pair of fusion categories} $(\mathcal A, \mathcal C, G, \Gamma, \tridb, \fizb, \trid, \fiz)$ consists of:
    \begin{itemize}
    \item faithful gradings 
        $\cA = \bigoplus_{g\in G} \cA_g$ and $ \cC =\bigoplus_{k\in \Gamma} \cC_k$,
    \item a matched pair of groups $(G, \Gamma, \tridb, \fizb)$,
    \item categorical left $\ku$-linear action $\trid$ of $\Gamma$ in $\cA$ and a categorical right $\ku$-linear action $\fiz$ of $G$ in $\cC$ such that $k \trid \cA_g = \cA_{k\tridb g}$ and $\cC_k \fiz g= \cC_{k \fizb g}$,
    \item two collections of natural isomorphisms $\gamma=(\gamma^k)_{k\in \Gamma}$ and $\eta=(\eta^g)_{g\in G}$ given by
    \begin{align*}
        &\gamma_{A,A'}^k\colon k\trid (A\ot A') \xrightarrow{\quad\simeq\quad} (k\trid A)\ot (k\fizb\degf{A})\trid A', & A,A'\in\Irr(\cA),\\
        &\eta_{C,C'}^g\colon (C\ot C')\fiz g \xrightarrow{\quad\simeq\quad} C\fiz (\degf{C'}\tridb g)\ot C'\fiz g, & C,C'\in\Irr(\cC),
    \end{align*}
 \item collections $\gamma_{0}^k\colon k\trid \mathbf{1} \to \mathbf{1}$ and $\eta_{0}^g\colon \mathbf{1}\fiz g \to \mathbf{1}$, $k\in \Gamma$, $g\in G$,
\end{itemize}
 satisfying the diagrams \cite[Definition 4.1]{mp-nuestro}. Using this data, we can construct the \emph{bicrossed product fusion category} $\mathcal A \bowtie \mathcal C$, which is an exact factorization between $\mathcal{A}$ and $\mathcal{C}$, see \cite[Corollary 4.9]{mp-nuestro}. The category $\mathcal A \bowtie \mathcal C$ is $\mathcal A \boxtimes \mathcal C$ as an abelian category. To define the tensor product, we pick a skeleton for $\cA\boxtimes \cC$ (that always exists for fusion categories, see \cite[Remark 2.8.7]{EGNO-book}) and then we define 
    \begin{align*}
    A\bowtie C \ot A'\bowtie C' = A\ot \degf{C} \trid A'\bowtie C\fiz \degf{A'} \ot C', && A,A'\in\Irr(\mathcal{A}), C,C'\in\Irr(\cC). 
\end{align*} 

Here $A\bowtie C$ is the isomorphism class chosen for the object $A\boxtimes C$. The associativity is necessarily non-trivial (as we used a skeleton) and it is given in \cite[\S 4]{mp-nuestro}.

\subsection{Crossed extensions of fusion categories} In this subsection, we recall the construction of crossed extensions of fusion categories introduced in \cite{Natale-crosssed-action}. These are extensions of a fusion category $\cC$ by $\operatorname{Rep} G$ (with $G$ a finite group) with respect to the $\operatorname{Rep} G$-module category $\vect$ of finite-dimensional vector spaces.

Let $(G,\Gamma, \tridb, \fizb)$ be a matched pair of groups. A \emph{$(G, \Gamma)$-crossed action} on a fusion category $\cC$ consists of a faithful grading on $\cC=\bigoplus_{k\in \Gamma} \cC_k$,  a categorical right $\ku$-linear action $\fiz$ of $G$ in $\cC$ such that $\cC_k \fiz g= \cC_{k \fizb g}$, and  a family of natural isomorphisms $\eta^g_{C, C'}: (C\ot C')\fiz g\to C\fiz (|C'|\tridb g)\ot C'\fiz g$, for all $g\in G$, $C, C'\in\cC$, such that the diagrams in \cite[Definition 4.1]{Natale-crosssed-action} commute. A crossed action give rise to an extension $\cC^{(G,\Gamma)}$ such that
\begin{align}\label{seq:sucesion-sonia}
    \operatorname{Rep} G\xhookrightarrow{} \cC^{(G,\Gamma)} \xrightarrow{} \cC.
\end{align}
The objects of $\cC^{(G,\Gamma)}$ are equivariant objects $(X,\{u_g\}_{g\in G})$ in $\cC^G$ with the monoidal structure given in \cite[Theorem 5.1]{Natale-crosssed-action}. Reciprocally,  any exact sequence $\operatorname{Rep} G\xhookrightarrow{} \mathcal{D} \xrightarrow{} \cC$ is equivalent to \eqref{seq:sucesion-sonia},  for some $(G, \Gamma)$-crossed action on $\cC$, see \cite[Theorem 1.1]{Natale-exact-sequence-G}. The category $\cC^{(G,\Gamma)}$ is called a $(G,\Gamma)$-\emph{crossed extension}.

\begin{remark}
    The results from Natale are stated in a more general hypothesis; they are valid for finite tensor categories without the semisimplicity, but by Remark \ref{remark:estable-by-semisimplicity} we can restrict ourselves to the semisimple case.
\end{remark}

\section{Equivalence of crossed extension and bicrossed products with $\vect_G$}\label{section:main-result}

In this section, we use the Basak-Gelaki correspondence to show that a crossed extension $\cC^{(G,\Gamma)}$ corresponds to a bicrossed product $\vect_G\bowtie \cC$. Since all extensions of $\cC$ by $\operatorname{Rep} G$ with respect to $\vect$ are given by crossed extensions, this shows that every exact factorization of the form $\vect_G\bullet\cC$ is a bicrossed product. Our main results are the following.

\begin{theorem}\label{teo:correspondencia-exact-con-bicrossed-product-con-vectG}
 Let $\cC$ be a fusion category and consider a $(G,\Gamma)$-crossed action on $\cC$. Then there exists a matched pair between $\vect_G$ and $\cC$ such that the crossed extension $\cC^{(G,\Gamma)}$ corresponds to a bicrossed product $\vect_G\bowtie \cC$ under the Basak-Gelaki correspondence.    
\end{theorem}

\begin{theorem}\label{teo:teorema-2}
Let $\cD = \vect_G\bullet \cC$ be an exact factorization of fusion categories. Then $\cD$ is tensor equivalent to a bicrossed product $\vect_G\bowtie \cC$.    
\end{theorem}

Our strategy for proving these theorems is as follows:

\begin{enumerate}[leftmargin=*,label=(\roman*)]
    \item\label{item:strategy-i} Start with a crossed extension data for $\cC$ and construct the tensor category $\cC^{(G,\Gamma)}$ and a matched pair between $\cC$ and $\vect_G$ to define $\vect_G\bowtie \cC$. 
    \item Apply the Basak-Gelaki correspondence to $\mathcal{B}=\vect_G\bowtie \cC$ to obtain an extension 
    \begin{align}\label{eq:extension-basak-gelaki}
\operatorname{Rep}G\xhookrightarrow{} \cB^*_\cC \xrightarrow{} \cC.
    \end{align}
    \item\label{item:strategy-iii}  Check that the extension \eqref{eq:extension-basak-gelaki} is equivalent to the original crossed extension
    \begin{align}\label{eq:extension-crossed-sonia}
\operatorname{Rep} G\xhookrightarrow{} \cC^{(G,\Gamma)} \xrightarrow{} \cC.
    \end{align}
    \item\label{item:strategy-iv}  Since every exact factorization $\vect_G \bullet \cC$ corresponds in a one-to-one way to an extension $\operatorname{Rep} G\xhookrightarrow{} \cD\xrightarrow{} \cC$
    and furthermore, every such extension is equivalent to one like \eqref{eq:extension-crossed-sonia}, then the exact factorization is equivalent to a bicrossed product $\vect_G\bowtie \cC$ by Theorem \ref{teo:correspondencia-exact-con-bicrossed-product-con-vectG}. 
\end{enumerate} Items \ref{item:strategy-i}-\ref{item:strategy-iii} prove Theorem \ref{teo:correspondencia-exact-con-bicrossed-product-con-vectG} and item \ref{item:strategy-iv} proves Theorem \ref{teo:teorema-2}.

The next result shows that there exists a matched pair of fusion categories associated with a $(G, \Gamma)$-crossed action.

\begin{prop}\label{prop:dato-sonia-para-armar-matched-pair}
    Let $(G,\Gamma, \tridb, \fizb)$ be a matched pair of groups, and $\cC$ be a fusion category with a $(G,\Gamma)$-crossed action $\fiz$. Then there exists a matched pair of fusion categories $(\vect_G, \cC, G, \Gamma, \tridb, \fizb, \trid, \fiz)$, with left action $k\trid g = k\tridb g$, for all $k\in\Gamma$, $g\in\Irr(\vect_G) = G$. The rest of the data are identity maps.
\end{prop}
\begin{proof}
The diagrams from \cite[Definition 4.1]{Natale-crosssed-action} are the same as \cite[Equations (11), (14), (15), (17), (19)]{mp-nuestro} and \cite[Lemma 4.3 - Equations (21),(22)]{mp-nuestro}. All other diagrams commute trivially because all the arrows involved are identities.
\end{proof}

Therefore, we have the bicrossed product fusion category $\cB=\vect_G\bowtie \cC$, by \cite[Corollary 4.9]{mp-nuestro}.
    
    \begin{lemma}\label{lemma:C-is-module-category-over-B}
        The category $\cC$ is a right $\cB$-module category with action $\bar{\ot}$ given by 
        \begin{align*}
            M\bar{\otimes} (g\bowtie C)=M\fiz g\otimes C,
        \end{align*} module associativity constraint given by the map 
        \begin{align*}
            m_{M, g\bowtie C, g'\bowtie C'} = & ((\eta^h_{M\fiz g, C})^{-1}\ot \id_{C'})((\mathbf{R}^2_{g,|C|\tridb g'})_M^{-1}\ot \id_{C\fiz g'}\ot\id_{C'})\\& \alpha^{-1}_{M\fiz g(|C|\tridb g'),C\fiz g', C'},
        \end{align*} and unit constraint given by $m^0_M=r_M(({\mathbf R}^0_M)^{-1}\ot \id_\uno)$, for all $C, C', M\in \cC$, $g, g'\in G$.
        \end{lemma}

Since $\cC$ is equipped with a $(G,\Gamma)$-crossed action, we can consider the crossed extension $\cC^{(G,\Gamma)}$. The next result shows the relation between this construction and the bicrossed product $\cB=\vect_G\bowtie \cC$. 

\begin{theorem}\label{teo:eq-Natale}
        There is a tensor equivalence $\cB^*_{\cC}=\operatorname{End}_{\cB}(\cC)\simeq \cC^{(G,\Gamma)}$.
\end{theorem}
\begin{proof}
The abelian part of the equivalence is similar to the usual equivalence between the dual of the equivariantization and the semidirect product with a group (see for example \cite[Example 3.17]{Nik_morita}) but, in this situation, a bit more detail is necessary since the $G$-action induced in $\cC$ is different (in principle) from the action $\fiz$ that $\cC$ comes equipped with. Because of this, we provide the full details of this equivalence. First, we define the functors $ \mathbf{T}  :  \cB^*_{\cC} \to  \cC^{(G,\Gamma)}$ and $\mathbf{S}:  \cC^{(G,\Gamma)} \to  \cB^*_{\cC}$, and check that they are well defined and inverse to each other. Then we show that $\mathbf{S}$ has a tensor structure, and this concludes the proof. 

The functor $ \mathbf{T}  :  \cB^*_{\cC} \to  \cC^{(G,\Gamma)}$ is defined as $\mathbf{T}(\mathbf F,s)=(\mathbf{F}(\uno), \{u_g\}_{g\in G})$, where $u_g: \mathbf{F}(\uno)\fiz g\to \mathbf{F}(\uno)$ is given by $u_g=\mathbf{F}(r_\uno)\mathbf{F}(\eta_0^ g\ot\id_\uno)s^{-1}_{\uno, g\bowtie \uno}r_{\mathbf{F}(\uno)\fiz g}^{-1}$, and given a natural transformation $\epsilon: (\mathbf F,s)\to (\mathbf G, t)$,  $\mathbf{T}(\epsilon)=\epsilon_\uno$, where $\epsilon_1: (\mathbf F (\uno), \{u_g\}_{g\in G})\to (\mathbf G (\uno), \{v_g\}_{g\in G})$. The functor $\mathbf{S}:  \cC^{(G,\Gamma)} \to  \cB^*_{\cC}$ is defined by $\mathbf{S} (U,  \{u_g\}_{g\in G})= ( \mathbf{S}_U, s)$, where $\mathbf{S}_U: \cC\to \cC$ is given by $\mathbf{S}_U(C)= U\ot C$ and, for all $M, C\in \cC$, $g\in G$, $s_{M, g\bowtie C}:  \mathbf{S}_U(M\bar{\ot}g\bowtie C)\to  \mathbf{S}_U(M)\bar{\ot}(g\bowtie C)$ is given by 
\begin{eqnarray}\label{equation:def-s-F}
    s_{M, g\bowtie C}& =({\eta_{U, M}^ g}^{-1}\ot\id_C)(u^{-1}_{|M|\tridb g}\ot\id_{M\fiz g}\ot \id_C)\alpha^{-1}_{U, M\fiz g, C},
\end{eqnarray} and given a morphism $f: (U,  \{u_g\}_{g\in G})\to (U',  \{u'_g\}_{g\in G})$,  $\mathbf{S}(f)=\epsilon$, where $\epsilon : ( \mathbf{S}_U, s)\to ( \mathbf{S}_{U'}, s')$ is given by $\epsilon_C=f\otimes \id_C$.

    Now, we check that $\mathbf{T}$ is well-defined, that is, $(\mathbf{F}(\uno), \{u_g\}_{g\in G})$ is an equivariant object. Since the category $\cC$ is a $\cB$-module category by Lemma \ref{lemma:C-is-module-category-over-B}, then it is a $\vect_G$-module category. By Lemma \ref{lemma:G-cat-G-functor}, $\cC$ is a $G$-category with induced action $M\overline{\fiz} g = M\fiz g \ot \uno$, and $\mathbf{F}$ is a $G$-functor. The $G$-categories $(\cC,\fiz)$ and $(\cC,\overline{\fiz})$ are equivalent through the $G$-functor $(\mathbf{Id}, r_{-\fiz g})$. By \cite[Equation (17)]{mp-nuestro}, the pair $(\uno,\eta_0^g)$ is an equivariant object in $(\cC,\fiz)$. Applying Lemma \ref{lemma:G-funtor-preserva-obj-equivariant} to this object with the $G$-functor $(\mathbf{Id}, r_{-\fiz g})$, we get that $(\uno, \eta^g_0 r_{\uno\fiz g}) = (\uno, r_{\uno} (\eta^g_0 \ot \id_{\uno}))$ is an equivariant object in $(\cC,\overline{\fiz})$. Applying again Lemma \ref{lemma:G-funtor-preserva-obj-equivariant} to the $G$-functor $(\mathbf{F},s)$, we get that $(\mathbf{F}(\uno), \mathbf{F}(r_\uno)\mathbf{F}(\eta_0^ g\ot\id_\uno)s^{-1}_{\uno, g\bowtie \uno})$ is an equivariant object in $(\cC,\overline{\fiz})$. Finally, again by Lemma \ref{lemma:G-funtor-preserva-obj-equivariant} now applying to $(\mathbf{Id}, r_{-\fiz g}^{-1})$ gives the desired equivariant object. 
    
    It is straightforward to see that $\epsilon_\uno$ is a morphism of equivariant objects, that is, $\epsilon_\uno u_g=v_g(\epsilon_\uno \fiz g)$, for all $g\in G$, by using the naturality of $r$ and $\epsilon$, and that $\epsilon: \mathbf F\to \mathbf G$ is a morphism of $\mathcal B$-module functors.

    Next, we prove that $\mathbf{S}$ is well-defined. To see that $( \mathbf{S}_U, s)$ is a $\cB$-module functor, we use the naturality of $\alpha_{U, -, \uno}$,  \cite[Lemma 4.3 - Equation (20)]{mp-nuestro}, \cite[Property (3.1.7)]{EGNO-book}, and that $(U, \{u_g\}_{g\in G})$ is a $G$-equivariant object. It follows from $\alpha$ and $\eta^g$ being natural transformations and $f: (U,  \{u_g\}_{g\in G})\to (U',  \{u'_g\}_{g\in G})$ being  a morphism of equivariant objects that the natural transformation $\epsilon$ is a morphism of $\cC$-module functor.

 Now, we show that $\mathbf{S}\circ \mathbf{T}\simeq \textrm{Id}_{\cB^*_{\cC}}$. We have 
 \begin{align*}
     (\mathbf{S}\circ \mathbf{T})(\mathbf{G}, s) =\mathbf{S}(\mathbf{G}(\uno), \{u_g\}_{g\in G})=(\mathbf{S}_{\mathbf{G}(\uno)}, t),  \,\,\text{where}
 \end{align*}  
\begin{align*}
   t_{M, g\bowtie C}= &({\eta^g_{\mathbf{G}(\uno), M}}^{-1}\ot \id_C)(r_{\mathbf{G}(\uno)\fiz (|M|\tridb g)}\ot\id_{M\fiz g}\ot \id_C)\\
   &(s_{\uno, (|M|\tridb g)\bowtie \uno}\ot\id_{M\fiz g}\ot \id_C ) (\mathbf{G}({\eta^{|M|\trid g}_0}^{-1}\ot\id_\uno)\ot\id_{M\fiz g}\ot \id_C)\\
   &(\mathbf{G}((r_\uno)^{-1})\ot\id_{M\fiz g}\ot \id_C)\alpha^{-1}_{\mathbf{G}(\uno), M\fiz g, C}.
\end{align*}

Then, we define the natural isomorphism $\nu: ( \mathbf{G},s)\to ( \mathbf{S}_{\mathbf{G}(\uno)},t)$ by 
\begin{align*}
\nu_C =&((\mathbf{R}^0_{G(\uno)})^{-1}\ot\id_C)s_{\uno, e\bowtie C}\mathbf{G}(\mathbf{R}^0_\uno\ot\id_C)\mathbf{G}(\ell_C^{-1}).
    \end{align*}

The condition that $\nu$ is a morphism of $\cC$-module functors, that is, $(\nu_M\bar{\ot}\id_{g\bowtie C})$ $s_{M, g\bowtie C}=t_{M,g\bowtie C}\nu_{M\bar{\ot}g\bowtie C}$, follows from the naturality of $s_{-,g\bowtie C}$, the naturality of $\eta^g_{-,M}$,  \cite[Lemma 4.3 - Equation (21)]{mp-nuestro},  \cite[Equation (14)]{mp-nuestro}, \cite[Equation (17)]{mp-nuestro}, the naturality of $s_{-, e\bowtie M\fiz g\ot C}$,  the naturality of $\mathbf{R}^ 0$, Equation \eqref{action-eq2},  $(\mathbf{G}, s)$ being a $\cC$-module functor, 
 \cite[Equation (15)]{mp-nuestro}, \cite[Lemma 4.3 - Equation (21)]{mp-nuestro}, the definition of $m$,  the structure of module functor of $G$ given by $s$, and  the naturality of $s_{\uno,-}$.

 We check now that $\mathbf{T}\circ \mathbf{S}\simeq \textrm{Id}_{\cC^{(G, \Gamma)}}$. We have that 
\begin{eqnarray*}
    (\mathbf{T}\circ \mathbf{S})(U, \{u_g\}_{g\in G})&=(\mathbf{S}_U, s)=(\mathbf{S}_U(\uno), \{u'_g\}_{g\in G}), \,\,\text{where}
    \end{eqnarray*}
    \begin{eqnarray*}
    u'_g & = (\id_U\ot m^0_\uno)(\id_V\ot\eta^g_0\ot\id_\uno)\alpha_{V, \uno\fiz g,\uno}(u_g\ot\id_{\uno\fiz g}\ot\id_\uno) (\eta^g_{U,\uno}\ot\id_\uno)r^{-1}_{(U\ot \uno)\fiz g}.
\end{eqnarray*}
In $\mathcal C$, we have the natural isomorphism $r$ between the two underlying objects. To see that the natural isomorphism $r_U^{-1}: (U,\{u_g\}_{g\in G})\to (\mathbf{S}_U(\uno), \{u'_g\}_{g\in G})$ is a map in $\cC^{(G, \Gamma)}$, we use the naturality of $r$,  \cite[Equation (15)]{mp-nuestro}, and  \cite[(2.13)]{EGNO-book}.

 The functor $\mathbf{S}$ is tensor with $J_{(U, \{u_g\}),(U', \{u'_g\})}: \mathbf{S}(U)\ot \mathbf{S}(U')\to \mathbf{S}(U\ot U')$ given by $J_{(U, \{u_g\}),(U', \{u'_g\})}(C):=\alpha^{-1}_{U, U', C}: U\ot (U'\ot C)\to  (U\ot U')\ot C$. In fact, the associativity in $\mathcal{B}^*_{\cC}$ is the identity and the associativity of $\cC^{(G,\Gamma)}$ is the associativity of $\cC$, then the diagram of the tensor structure of $(\mathbf{S},J)$ is the pentagon of the associativity of $\cC$.

\end{proof}

\begin{lemma}\label{lemma:equivalencia-sucesiones}
    There is an equivalence of extensions between
    \begin{align*}
     \operatorname{Rep}G\xhookrightarrow{\mathbf{\iota_1}} \cC^{(G,\Gamma)} \xrightarrow{\mathbf{F}_1} \cC &&  \text{ and} &&\operatorname{Rep}G\xhookrightarrow{\mathbf{\iota}_2} \cB^*_\cC \xrightarrow{\mathbf{F}_2}  \cC.
\end{align*} 
\end{lemma}

\begin{proof}
   We first describe the functors that appear in the statement. Let us fix the tensor functor, 
   \begin{align*}
       \operatorname{can}_{\cC}\colon \vect \to \cC, \operatorname{can}_{\cC}(V)= \oplus_{i=1}^{\dim V} \uno,
   \end{align*} and then we have that
\begin{align*}
    \iota_1(V,\{\rho(g)\}_{g\in G}) &= (\operatorname{can}_{\cC}(V),\{\operatorname{can}_{\cC}(\rho(g)) (\oplus_{i=1}^{\dim V}\eta_0^g)\}_{g\in G} ).
\end{align*}

The functor $\mathbf{F}_1: \cC^{(G,\Gamma)}\to \cC$ is the forgetful fuctor. 

We define
\begin{align*}
    \iota_2(V,\rho) & =(F_V, s),
\end{align*}
with the functor $F_V(C) = \operatorname{can}_{\cC}(V) \otimes C$ and $s_{M, g\bowtie C}$ is given by $({\eta_{\operatorname{can}_{\cC}(V), M}^ g}^{-1}\ot\id_C)((\operatorname{can}_{\cC}(\rho(g)^{-1})(\oplus_{i=1}^{\dim V}{\eta_0^{|M|\tridb g}}^{-1})\ot\id_{M\fiz g}\ot \id_C)\alpha^{-1}_{\operatorname{can}_{\cC}(V), M\fiz g, C}$,  for all $M, C\in\cC, g\in G$. 
   
   The functor $\mathbf{F}_2$ is given by $\mathbf{F}_2=\mathbf{H}\circ \mathbf{I}\circ \mathbf{G}$, where each functor is defined as follows. Let $\cC_{\widetilde{\ot}}$ be the $\cC$-module category $\cC$ with the action restricted from the action of $\cB$ on $\cC$, more precisely, $M\widetilde{\ot} X=M\fiz e\ot X$ and $\widetilde{m}_{M, X, Y}$ is given by $m_{M, e\bowtie X, e\bowtie Y}(\id_{M\fiz e}\ot (\mathbf{R}^0_{X})^{-1}\ot\id_Y)$, for all $M, X\in \cC$. The tensor functor $\mathbf{G}:\cB^*_\cC\to \cC^*_{{\cC}_{\widetilde{\otimes}}}$ is given by $\mathbf{G}(\mathbf{F}, s)=(\mathbf{F},\overline{s})$, with $\overline{s}_{M,X}=s_{M, e\bowtie C}$, for all $M, X\in\cC$, and the monoidal structure of $\mathbf{G}$  is the identity.
     Notice that $\cC_{\widetilde{\ot}}\simeq \cC$ as $\cC$-module categories, where $\cC$ is viewed as a $\cC$-module category with the regular action. Therefore, there is a tensor equivalence  $\mathbf{I}: \cC^*_{{\cC}_{\widetilde{\otimes}}} \to \cC^*_{\cC}$ explicitly given as  $\mathbf{I}(\mathbf{F}, s)=(\mathbf{F}, s')$, $s'_{M, X}=((\mathbf{R}^0)^{-1}\ot \id_X)s_{M,X}\mathbf{F}(\mathbf{R}^0_M\ot\id_X)$, for all $M, X\in \cC$, and the monoidal structure of $\mathbf{I}$ is the identity.
   The tensor functor $\mathbf{H}: \cC^*_{\cC}\to \cC$ is the canonical functor given by $\mathbf{H}(\mathbf{F}, t)=\mathbf{F}(\uno)$ with monoidal structure $\mathbf{H}^2_{(\mathbf{F}, t), (\mathbf{F}', t')}=\mathbf{F}(\ell_{\mathbf{F}'(\uno)})t^{-1}_{\uno, \mathbf{F}'(\uno)}$.

The good definition of $\iota_2$ is the same as the good definition of the functor $\textbf{S}$, defined in Theorem \ref{teo:eq-Natale}, compared to Equation \ref{equation:def-s-F}. It is clear that the left square of the following diagrams commutes 
\begin{center}
    \begin{tikzcd}
    \operatorname{Rep} G\ar[d, equal]\ar[r, hook, "{\mathbf{\iota}_1}"]  & \cC^{(G, \Gamma)} \ar[r, "{\mathbf{F}_1}"]\ar[d, "\mathbf{S}"] & \cC\ar[d, equal] \\
    \operatorname{Rep} G \ar[r, hook, "{\mathbf{\iota}_2}"] & \cB^*_{\cC}\ar[r, "{\mathbf{F}_2}"] & \cC.
\end{tikzcd}
\end{center}

Let us see that the right square commutes up to a monoidal natural isomorphism. Since $(\mathbf{F}_2\circ \mathbf{S})(X, \{u_g\}_{g\in G})=X\otimes \uno$ and $\mathbf{F}_1(X, \{u_g\}_{g\in G})=X$, we have the natural isomorphism $r: \mathbf{F}_2\circ\mathbf{S}\to \mathbf{F}_1$. The monoidal structure of the functor $\mathbf{F}_2\circ \mathbf{S}$ is given by $(\mathbf{F}_2\circ \mathbf{S})^2_{(X, \{u_g\}_{g\in G}),(Y, \{v_g\}_{g\in G})}=\alpha^{-1}_{X, Y, \uno}(\id_X\otimes \ell_{Y\otimes\uno})\overline{s}^{-1}_{\uno, Y\otimes\uno}$, where $s$ is the module structure of the functor $\mathbf{F}_X$ defined in \eqref{equation:def-s-F}. To see that $r$ is monoidal, we use the triangle equation, naturality of $\alpha$, naturality of $\ell$,  \cite[Property (2.13)]{EGNO-book},  \cite[Equation (22)]{mp-nuestro}, and $(X,\{u_g\}_{g\in G})$ being an equivariant object.
\end{proof}

\subsection*{Proof of Theorem \ref{teo:correspondencia-exact-con-bicrossed-product-con-vectG}}

\begin{proofw}
Given the matched pair given in Proposition \ref{prop:dato-sonia-para-armar-matched-pair}, then we have a bicrossed product of fusion categories $\cB = \vect_G \bowtie \cC$; see \cite[Theorem 4.6]{mp-nuestro}. By \cite[Prop. 4.7]{mp-nuestro}, we have canonical inclusion functors $\vect_G\xhookrightarrow{\nu} \cB \xhookleftarrow{\kappa} \cC$. Now, as described in Subsection \ref{subsubsect:mapa-psi}, we apply $\psi_{\vect_G,\vect}$ to the exact factorization $\cB = \vect_G \bowtie \cC$ to get an exact sequence 
\begin{align*}\label{eq:exact-sequence-proof-nuestro}
    \operatorname{Rep} G\simeq {\vect_G}^*_{\vect}\xhookrightarrow{\iota_{{\vect_G}^*_{\vect}}} \cB^{*}_{\cC} \xrightarrow{\kappa^*} \cC^*_\cC\simeq\cC.
\end{align*}
It is straightforward to see that this sequence is equivalent to the sequence
  \begin{align*}
 \operatorname{Rep}G\xhookrightarrow{\mathbf{\iota}_2} \cB^*_\cC \xrightarrow{\mathbf{F}_2}  \cC,
\end{align*} 
defined in Lemma \ref{lemma:equivalencia-sucesiones}. By the same lemma, this sequence is equivalent to the sequence  \begin{align*}
     \operatorname{Rep}G\xhookrightarrow{\mathbf{\iota_1}} \cC^{(G,\Gamma)} \xrightarrow{\mathbf{F}_1} \cC.
\end{align*}
\end{proofw}

\subsection*{Proof of Theorem \ref{teo:teorema-2}} 
\begin{proofw}
Let $\vect_G\xhookrightarrow{}\mathcal{D}\xhookleftarrow{}\cC$ be an exact factorization of fusion categories. Then, by \cite[Theorem 5.1]{generalization-G-nonsemisimple}, there exists an exact sequence of fusion categories 
\begin{equation}\label{eq:exact-sequence-proof-5-2}
    \operatorname{Rep} G\xhookrightarrow{} \mathcal{D}^*_{\mathcal{\cC}} \xrightarrow{} \cC
\end{equation}
associated to $\cD$. It follows from \cite[Theorem 1.1]{Natale-exact-sequence-G} that Sequence \eqref{eq:exact-sequence-proof-5-2} is equivalent to 
\begin{align}\label{eq:exact-sequence-proof-5-2-sonia}
    \operatorname{Rep} G\xhookrightarrow{\iota_1} \cC^{(G,\Gamma)} \xrightarrow{\mathbf{F}_1} \cC,
\end{align}
where $\cC^{(G,\Gamma)}$ is a crossed extension, and the functors $\iota_1$ and $\mathbf{F}_1$ are described in the proof of Lemma \ref{lemma:equivalencia-sucesiones}. By Theorem \ref{teo:correspondencia-exact-con-bicrossed-product-con-vectG}, Sequence \eqref{eq:exact-sequence-proof-5-2-sonia} is equivalent to the image of $\psi_{\vect_G,\vect}$ on to a bicrossed product $\vect_G\bowtie \cC$. By \cite[Corollary 5.6]{EG-exact-sequence}, the map $\psi_{\vect_G,\vect}$ is one-to-one, hence $\cD$ is equivalent to $\vect_G\bowtie \cC$.
\end{proofw}

\begin{remark}
In \cite[\S5.1]{mp-nuestro}, we characterize the bicrossed products between a Tambara-Yamagami fusion category and a pointed fusion category. Then the previous proposition shows that these data characterize all exact factorizations between $\vect_G$ and a Tambara-Yamagami fusion category.
\end{remark}

\bibliographystyle{alpha}
\bibliography{conection-natale-biblio}

\end{document}